\documentclass[11pt]{amsart}
\usepackage{amscd}
\usepackage{amsthm}
\usepackage[centertags]{amsmath}
\usepackage{amsfonts}
\usepackage{newlfont}
\usepackage{graphicx}
\usepackage[usenames]{color}
\usepackage{amsfonts, amssymb,accents}
\usepackage{mathrsfs}
\usepackage{latexsym}
\usepackage{bbold}
\usepackage{enumitem}
\usepackage{tikz}\usepackage{verbatim}
\usepackage{tabulary,tabularx}
\usepackage[all]{xy}
\usepackage{tikz}
\usetikzlibrary{positioning}
\usetikzlibrary{shapes.multipart}

\newtheorem{theorem}{Theorem}[section]

\newtheorem{proposition}[theorem]{Proposition}
\newtheorem{corollary}[theorem]{Corollary}
\newtheorem{lemma}[theorem]{Lemma}
\newtheorem{remark}[theorem]{Remark}
\newtheorem{definition}[theorem]{Definition}

\newtheorem{question}[theorem]{Question}

\newcommand{\zN}{\mathbb N}

\newcommand{\sub}{\subseteq}
\newcommand{\zT}{\mathbb T}

\newcommand{\f}{\frac}

\newcommand{\AP}{\mathcal {AP}}
\newcommand{\kAP}{\mathcal {AP}_b}
\newcommand{\A}{\mathcal {F}}

\newcommand{\interior}[1]{\accentset{\smash{\raisebox{-0.12ex}{$\scriptstyle\circ$}}}{#1}\rule{0pt}{2.3ex}}
\renewcommand{\l}{\left(}
\renewcommand{\r}{\right)}

    \title{Arithmetic progressions and chaos in linear dynamics}

\begin{document}

\begin{abstract}
We characterize chaotic linear operators on reflexive Banach spaces in terms of the existence of long arithmetic progressions in the sets of return times. 
To achieve this, we study $\A$-hypercyclicity for a family of subsets of the natural numbers associated to the existence of arbitrarily long arithmetic progressions. We investigate their connection with different concepts in linear dynamics. 
\end{abstract}

\author{Rodrigo Cardeccia, Santiago Muro}
\address{DEPARTAMENTO DE MATEM\'ATICA - PAB I,
	FACULTAD DE CS. EXACTAS Y NATURALES, UNIVERSIDAD DE BUENOS AIRES, (1428) BUENOS AIRES, ARGENTINA AND IMAS-CONICET} \email{rcardeccia@dm.uba.ar} 
\address{FACULTAD DE CIENCIAS EXACTAS, INGENIERIA Y AGRIMENSURA, UNIVERSIDAD NACIONAL DE ROSARIO, ARGENTINA AND CIFASIS-CONICET}
\email{muro@cifasis-conicet.gov.ar}

\thanks{Partially supported by ANPCyT PICT 2015-2224, UBACyT 20020130300052BA, PIP 11220130100329CO and CONICET}

\maketitle
\section { Introduction }
A linear operator $T$ is said to be hypercyclic provided that there is $x\in X$ such that $Orb_T(x):=\{T^n(x):n\in \zN\}$ is dense in $X$ and chaotic if it is hypercyclic and has dense periodic points. The notion of chaos was introduced by Devaney \cite{Dev89} and developed by Godefroy and Shapiro \cite{GodSha91} in the context of linear dynamics. Since then it was one of the most important concepts in the dynamics of linear operators.  Linear dynamics has experienced a lively development in the last decades, see \cite {BayMat09,GroPer11}. For instance we know that every infinite dimensional and separable Banach space supports a hypercyclic operator \cite {ansari1997existence,bernal1999hypercyclic} while there are Banach spaces without chaotic operators \cite {BonMarPer01}, there are hypercyclic operators $T$ such that $T\oplus T$ are no longer hypercyclic \cite {dlRRea09}, etc.  Over the last years attention was given to \textit {frequent hypercyclicity} \cite{BayGri06} and more recently to $\A$\textit{-hypercyclicity} \cite{Bes16,BesMen19,BonGro18,bonilla2020frequently}, for more general families $\A$ of subsets of $\mathbb N$. 

Given a  \textit {hereditary upward} family $\A\sub \mathcal P(\zN)$ (also called Furstenberg family) we say that an operator is $\A$-hypercyclic if there is $x\in X$ for which the sets $N_T(x,U):=\{n\in\zN: T^n(x)\in U\}$ of return times belong to $\A$. Thus, for example, if we consider $\A_{\neq \emptyset}$, the family of non empty sets, $\A_{\neq \emptyset}$-hypercyclicity is simply hypercyclity and if $ \underline{\mathcal{D}}$ denotes the family of sets with positive lower density,  then $ \underline{\mathcal{D}}$-hypercyclicity is frequent hypercyclicity. Over the last years several notions of $\A$-hypercyclicity were introduced such as upper frequent hypercyclicity \cite{shk09}, reiterative hypercyclicity \cite{Bes16} and more recently piecewise-syndetic hypercyclicity \cite{Pui17}.

 It is known that there are frequently hypercyclic operators that are not chaotic \cite {BayGri07} and it was recently shown that there are chaotic operators that are not frequently hypercyclic \cite{Men17}. However, the connection between frequently hypercyclicity and chaos is still not very well understood. 
 The following question was posed by Bonilla and Grosse-Erdmann  \cite{BonGro18}.
 
\begin{question} \label{Pregunta1} Does there exist a hereditery upward family $\mathcal F$ such that $\A$-hypercyclicity is equivalent to chaos?
 \end{question}
A related (weaker) question is the following.
\begin{question} \label{Pregunta2} Is it possible to characterize chaos in terms of the behavior of a single orbit?
 \end{question}
In the present note, we introduce a  notion of $\A$-hypercyclicity related to the existence of long arithmetic progressions 
and study its connection with chaos and other concepts in linear dynamics.  We answer Question \ref{Pregunta1} affirmatively on separable reflexive Banach spaces and Question \ref{Pregunta2} on arbitrary separable Fr\'echet spaces.


The motivation to study the relationship between arithmetic progressions and chaotic operators is simple: if $T$ is a chaotic operator and $U$ is a nonempty open set then the existence of periodic points on $U$ imply that $N_T(x,U)$ must have arbitrarily long  arithmetic progressions for any hypercyclic vector $x$.

The study of sets having arbitrarily long arithmetic progressions (or sets in $\AP$) had a great development over the last century and is a central task in number theory and additive combinatorics. 
On the other hand, there isn't, up to our knowledge, a systematic investigation on sets having arbitrarily long arithmetic progressions with fixed common difference.   Nevertheless, as we shall see, these sets play an important roll in linear dynamics. We will denote by $\kAP$  to this family of subsets, and we will use it to answer Question \ref{Pregunta1}   for weak$^*$-weak$^*$ continuous operators: such an operator  is chaotic if and only if it is $\kAP$-hypercyclic (Theorem \ref{main AP*}). For arbitrary operators the family $\kAP$ can still be used to characterize chaos in terms of a single orbit: $T$  is chaotic if and only if there is $x\in X$ such that for every nonempty open set $U$, the return times set $N(x,U)$ contains  a subsequence $(n_k)_k\in \kAP$ for which the set $\{T^{n_k}(x)\}$ is weakly precompact. 
As a corollary, we obtain a Transitivity Theorem (Theorem \ref{transitive theorem}) for chaotic operators.

The paper is organized as follows. In Section \ref{section preliminares} we fix notation and recall some facts about hereditary upward families and $\A$-hypercyclicity. 
In Section \ref{section salap} we study $\kAP$ hypercyclic operators, operators having dense small periodic sets and their connection to chaos. We prove that these concepts are equivalent for weak$^*$-weak$^*$ continuous operators (Theorems \ref{main AP*} and \ref{Caoticos}) and we show the existence of an $\kAP$-hypercyclic weighted shift on $c_0$ which is not chaotic. We also prove that hypercyclic weighted shifts with dense small periodic sets are chaotic (Theorem \ref{Backward c_0}) and that $\kAP$-hypercyclic operators cannot have isolated points in the spectrum (Corollary \ref{espectro}).

\section{Preliminaries}\label{section preliminares}

A family $\mathcal F\sub \mathcal P (\zN)$ is said a \textit{hereditary upward family} or a \textit {Furstenberg family} if $A\sub B$ and $A\in \mathcal F$, then $B\in \mathcal F$ (see for example \cite{Aki97,BonGro18}). Given a Furstenberg family $\mathcal F$ we will say that $T$ is $\mathcal F$-hypercyclic provided that there is $x\in X$ such that for every open set $U$, $N_T(x,U):=\{n: T^n(x)\in U\}\in \mathcal F.$ Such an $x$ is called an $\mathcal F$-hypercyclic vector. 

The following hereditary upward families and notions of $\A$-hypercyclicity are the most widely studied in the literature:
\begin{enumerate}
\item $A$ is said to have positive lower density (or $A\in  \underline{\mathcal{D}}$) if $\underline {dens} (A):=$ $\liminf_n \# \f{\{k\leq n: k\in A\}}{n}>0$  and an operator is said to be frequently hypercyclic if $T$ is $ \underline{\mathcal{D}}$ hypercyclic. 
\item $A$ is said to have positive upper density (or $A\in \A_{ud}$) if $\limsup_n \# \f{\{k\leq n: k\in A\}}{n}>0$  and an operator is said to be $\mathcal U$-frequently hypercyclic if $T$ is $\A_{ud}$-hypercyclic.
\item $A$ is said to have positive Banach upper density (or $A\in\A_{bd}$) if $\lim_n\limsup_k \f{\# A\cap [k,k+n]}{n}>0$ and an operator is said to be reiterative hypercyclic if $T$ is $\A_{bd}$-hypercyclic.
\end{enumerate}
 


A hereditary upward family is said to be \textit {upper} provided that $\emptyset\notin \A$ and $\A$ can be written as 
$$\bigcup_{\delta\in D} A_\delta,\quad\textrm{with }\; \A_\delta=  \bigcap_{m\in M} \A_{\delta,m},$$
where $M$ is countable and such that the families $\A_{\delta,m} $ and $\A_{\delta}$ satisfy
\begin{itemize}
\item each $\A_{\delta,m}$ is \textit{finitely hereditary upward}, that means that for each $A\in \A_{\delta,m}$, there is a finite set $F$ such that $F\cap A\sub B,$ then $B\in \A_{\delta,m}$;
\item $\A_\delta$ is \textit{uniformly left invariant}, that is, if $A\in\A$ then there is $\delta$ such that for every $n$, $A-n\in \A_\delta$. 
\end{itemize}
 The families $\A_{\neq \emptyset}$, $\A_{ud},\A_{bd}$ are upper while $ \underline{\mathcal{D}}$ is not upper (see \cite{BonGro18}).


\begin{theorem}[Bonilla-Grosse Erdmann \cite{BonGro18}]\label{equivalencias upper}
Let $\A$ be a an upper hereditary upward family and $T$ be a linear operator on a separable Fr\'echet space. Then the following are equivalent:
\begin{enumerate}
    \item For any open set $V$ there is $\delta$ such that for any open set $U$ there is $x\in U$ with $N_T(x,U)\in \A_\delta.$
    \item For any open set $V$ there is $\delta$ such that for every $U$ and $m$ there is $x\in U$  with $N_T(x,U)\in \A_{\delta,m}$.
    \item The set of $\A$-hypercyclic points is residual.
    \item $T$ is $\A$-hypercyclic.
\end{enumerate}

\end{theorem}

\section{ $\kAP$-hypercyclic operators  and chaotic operators}\label{section salap}
In this section we study $\kAP$-hypercyclic operators and their relationship with chaotic operators. Our main result is the following theorem, which shows that the weak$^*$-weak$^*$ continuous chaotic operators are exactly the $\kAP$-hypercyclic operators. We will also show in Theorem \ref{Without periodic sets} that the assumption on weak$^*$-weak$^*$ continuity cannot be dropped, by exhibiting an $\kAP$-hypercyclic operator on $c_0$ that is not chaotic.
\begin{theorem}\label{main AP*}
Let $X$ be a separable Banach space which is a dual space and let $T$ be a weak$^*$-weak$^*$ continuous operator on $X$. Then $T$ is chaotic if and only if there exists $x\in X$ such that for each nonempty open set $U$,  $N_T(x,U)$ contains arbitrarily long arithmetic progressions of common difference $k$, for some $k\in\zN$.
\end{theorem}
Note that the above equivalence holds for arbitrary operators on reflexive spaces.
Theorem \ref{main AP*} is a direct consequence of Theorem \ref{Caoticos} below. 
Let us first define the Furstenberg family $\kAP$.
Recall that the arithmetic progression of length $m+1$ ($m\in\mathbb N$), common difference $k\in\mathbb N$ and initial term $a\in\mathbb N$ is the subset of $\mathbb N$ of the form $\{a,a+k,a+2k,\dots,a+mk\}$.

\begin{definition}\label{def salap}\rm
    We will denote by $\kAP$ to the family of subsets of the natural numbers that contain {arbitrarily long arithmetic progressions of bounded common difference} (i.e. there are arbitrarily long arithmetic progressions of common difference bounded by  $k$, for some fixed $k\in\mathbb N$).
\end{definition}
The family $\kAP$ is an upper Furstenberg family: it is the union of the families $(\kAP)_{n}$ of subsets having arbitrarily long arithmetic progressions with fixed step $n$,  and $(\kAP)_{n}$ is the intersection of the families $(\kAP)_{n,m}$ of subsets having arithmetic progressions fixed step $n$ with length $m$.
The next proposition is thus a consequence of Theorem \ref{equivalencias upper}.


\begin{proposition}\label{equivalencias salap}
Let $T$ be an operator on a Fr\'echet space. Then the following assertions are equivalent.
\begin{enumerate}[label=(\arabic*)]
\item $T$ is hypercyclic and  every hypercyclic vector is $\kAP$-hypercyclic.
\item There is an $\kAP$-hypercyclic vector.
\item $T$ is hypercyclic and for every open set $U$ there is $k$ such that for every $m$ , $\bigcap_{j=1}^m T^{-jk}(U)\neq \emptyset.$
\item \label{salap transitivo} For every open sets $U$ and $V$ there is $k$ such that for every $m$ there are $k_m$ and $x\in U$ with $T^{k_m+jk}(x)\in V$ for every $0\leq j\leq m$. 

\item  The set of $\kAP$-hypercyclic vectors is residual. 
\end{enumerate}
\end{proposition}
In \cite{Men17} it was shown that chaotic operators are reiteratively hypercyclic. The proof given there  essentially proves the following.
\begin{proposition}\label {chaotic entonces salap}
Let $T$ be a chaotic operator. Then $T$ is $\kAP$-hypercyclic.
\end{proposition}

On the other hand, there are subsets of the natural numbers (for instance the square free numbers) that have positive lower density but do not belong to $\kAP$. So we cannot conclude that frequently hypercyclic operators are $\kAP$-hypercyclic. There are also chaotic operators that are not  upper frequently hypercyclic \cite{Men17}. Therefore there are $\kAP$-hypercyclic operators that are not upper frequent hypercyclic. Moreover, since there are frequently hypercyclic operators on Hilbert spaces that are not chaotic \cite[Section 6.5]{BayMat09}, by Theorem \ref{main AP*} frequent hypercyclicity does not imply $\kAP$-hypercyclicity.

Since $\kAP$ sets have positive upper Banach density and since reiterative hypercyclic operators are weakly mixing \cite{Bes16} we have that $\kAP$-hypercyclic operators are weakly mixing. 

\begin{proposition}
    Let $T$ be a $\kAP$-hypercyclic operator. Then $T$ is reiterative hypercyclic. In particular $T$ is weakly mixing.
\end{proposition}


In order to prove Theorem \ref{main AP*}, we need the  concept of dense small periodic sets, which is a natural generalization of density of periodic points. The notion was introduced by Huan and Ye in \cite{HuaYe05} for non linear dynamics on compact spaces. 
We will say that a subset $Y$ is a periodic set for $T$ if $T^k(Y)\subset Y$ for some $k>0$.
\begin{definition}
    A  mapping $T$ has dense small periodic sets provided that for every open set $U$ there is a closed periodic set $Y\sub U$.
\end{definition}

\begin{proposition}\label{dense small implica AP*}
An operator $T$ has dense small periodic sets if and only if for every nonempty open set $U$ there is $k$ such that $\bigcap_{j=1}^\infty T^{-jk}(U)\neq \emptyset.$ 

In particular if $T$ is hypercyclic and has dense small periodic sets then it is $\kAP$-hypercyclic.
\end{proposition}
\begin{proof}
Let $U$ be an open set and consider $V\sub U$ such that $V\sub \overline V\sub U$. Let $x\in \bigcap_{j=1}^\infty T^{-jk}(V)$. Then the set $Y=\overline{Orb_{T^k}(x)}$ is a closed subset of $U$ which satisfies  $T^k(Y)\subset Y$. Reciprocally given an open set $U$ and $Y\sub U$ a closed  subset which is $T^k$-invariant, every $x\in Y$  belongs to $\bigcap_{j=1}^\infty T^{-jk}(U).$

The last assertion follows from Proposition \ref{equivalencias salap}.
\end{proof}
In  subsection \ref{weighted shifts} we will present an example of an $\kAP$-hypercyclic operator that does not have dense small periodic sets. 

The next lemma, which  is purely linear as it exploits  the linearity of both the operator and the space, is an important ingredient for the proof of the main theorem.
\begin{lemma}\label{lema w-compact periodic set tienen ptos periodicos}
Let $Y$ be a $k$-periodic set for an operator $T$ on a Fr\'echet space $X$ such that either
\begin{enumerate}[label=\roman*)]
    \item $Y$ is weakly compact or
    \item $X$ is a dual space, $Y$ is weak$^*$-compact and $T$ is weak$^*$-weak$^*$  continuous. 
\end{enumerate}
Then there is a $k$-periodic vector in $\overline{co(Y)}^\tau$, where $\tau$ denotes 
weak or weak star topology, respectively.
\end{lemma}
\begin{proof}
The proof is an elementary application of the Schauder-Tychonoff fixed point Theorem for locally convex spaces \cite{Tyc35}.  


Let $Y$ be a $k$-periodic set. Then $\overline {co (Y)}^{\tau}$ is $T^k$-invariant. Moreover, $\overline {co (Y)}^{\tau}$ is $\tau$-compact (by either the Krein-\v Smulian Theorem \cite{KreiSmul40} or \cite[Chapter II, 4.3]{SchaeferTVS}).
Therefore, the  Schauder-Tychonoff Theorem assures the existence of a fixed point of $T^k$ in $\overline {co (Y)}^{\tau}$. This fixed point is a $k$-periodic vector for $T$.
\end{proof}

\begin{proposition}\label{prop AP* implica ptos periodicos en el bidual}
Let $T$ be a hypercyclic operator on a separable Banach space. Consider the following statements.
\begin{itemize}
    \item[i)] $T$ is $\kAP$-hypercyclic.
    \item[ii)] For each nonempty open set $U\subset X$ there is a closed periodic set of $T^{**}$ contained in $\interior{\overline{U}^{w*}}\subset X^{**}$.
    \item[iii)] For each nonempty open set $U\subset X$ there is $k$ such that $\displaystyle\bigcap_{j=1}^\infty (T^{**})^{-jk}\l\interior{\overline U^{\omega^*}}\r\neq \emptyset.$
    \item[iv)] The norm closure of the periodic points of $T^{**}$ contains $X$.
    \end{itemize}
Then $i)\Rightarrow ii)\Leftrightarrow iii)\Leftrightarrow iv).$
\end{proposition}
\begin{proof}
$i)\Rightarrow ii)$ 
Let $U\subset X$ be an open set and $V=B_r(x_0)$ so that $\overline V\sub U$, $\overline V^{\omega^*}\subset \interior{\overline U^{\omega^*}}$ is a weak$^*$-compact set in $X^{**}$. Let  $x\in V$ such that $N(x,V)\in \kAP$. Thus, there are $k\in \zN$ and a sequence $(a_n)_n$, such that $T^{a_n +ik}(x)\in V$ for every $i\leq n$. There is a weak$^*$-limit point $y\in \overline V^{\omega^*}$ of the sequence $(T^{a_n}(x))_n$

Let $Y=\overline{Orb_{(T^{**})^k}(y)}^{\omega^*},$ the weak$^*$-closure of the orbit of $y$ under $(T^{**})^k$. This set is clearly $(T^{**})^k$-invariant, so we only need to show that $Y\sub \interior{\overline U^{\omega^*}}$. It suffices to show that for every $m$, $(T^{**})^{km}(y)\in \overline V^{\omega^*}.$  Fix $m\in \zN$ and notice that $T^{**}$ since is weak$^*$-weak$^*$-continuous then 
 $(T^{**})^{km}(y)$ is a weak$^*$-limit point of $((T^{**})^{a_{n}+km}(x))_n=(T^{a_{n}+km}(x))_n$. Since for any $n\ge m$, we have that $T^{a_{n}+km}(x)\in  V$, we conclude that $T^{mk}(y)\in \overline V^{\omega^*}.$

 $ii)\Leftrightarrow iii)$ Follows as the proof of Proposition \ref{dense small implica AP*}.
 
 $ii)\Leftrightarrow iv)$ Since any ball of the bidual $X^{**}$ centered at a point of $X$ contains a weak$^{*}$-compact periodic set, statement $iv)$ holds by Lemma \ref{lema w-compact periodic set tienen ptos periodicos}. The converse is immediate.
 
 \end{proof}

Note that, in particular, the above proposition proves Theorem \ref{main AP*} for reflexive spaces.
Let us see that we can push this argument a little further.

 

In \cite[Proposition 3.2]{HuaYe05} Huang and Ye studied the relationship between compact dynamical systems having dense small periodic sets and sets $N_f(x,U)$ having arbitrarily long arithmetic progressions with fixed step (see also \cite{Li11}).
The following lemma is a generalization of their result to dynamical systems on infinite dimensional  spaces.

 \begin{lemma}\label{AP* implica dense small}
Let $X$ be a separable Banach space which is a dual space and let $T:X\to X$ be a  weak$^*$-weak$^*$ continuous (not necessarily linear) mapping. Then $T$ is  $\kAP$-hypercyclic if and only if $T$ is hypercyclic and has dense small periodic sets.
\end{lemma}
\begin{proof}
One implication is Proposition \ref{dense small implica AP*}. The proof of the converse is similar to the proof of $i)\Rightarrow ii)$ in Proposition \ref{prop AP* implica ptos periodicos en el bidual}.
 \end{proof}

We can now prove our main theorem, which can be restated as follows.
 \begin{theorem}\label{Caoticos}
Let $X$ be a separable  Banach space which is a dual space and let $T:X\to X$ be a weak$^*$-weak$^*$ continuous linear operator. The following assertions are equivalent:
\begin{enumerate}
    \item [i)] $T$ is  $\kAP$-hypercyclic;
    \item [ii)]$T$ is hypercyclic and has dense small periodic sets and
    \item [iii)]$T$ is chaotic.
\end{enumerate}
\end{theorem}
 \begin{proof}
 i)$\Longleftrightarrow $ ii) is Lemma \ref{AP* implica dense small}. iii)$\Longrightarrow $ ii) is immediate.
 

For ii)$\Longrightarrow$ iii)
let $U$ be an open set. We must show that $T$ has a periodic point in $U$. Consider  $V\sub U$ such that $V$ is open, convex, weak$^*$-precompact and such that $\overline V\sub U$. 
 Let $Y\sub V$ be a $k$-periodic set. Then by Lemma \ref{lema w-compact periodic set tienen ptos periodicos}, $T$ has a $k$-periodic point in
 $\overline {co (Y)}^{\omega^*}\subset U$. 
 \end{proof}
In a similar way we have.
\begin{corollary}\label{dense small w-compact sets}
Let $X$ be a Fr\'echet space and $T$ a linear operator that has  dense small weakly  compact periodic sets. Then $T$ has dense  periodic points.  
\end{corollary}

If we apply Proposition \ref{equivalencias salap} we obtain a transitivity theorem for chaotic operators.
\begin{corollary}[A transitivity Theorem for chaotic operators.]\label{transitive theorem}
Let $X$ be a separable  Banach space which is a dual space and let $T:X\to X$ be weak$^*$-weak$^*$ continuous. Then the following are equivalent:
\begin{enumerate}
    \item $T$ is chaotic;
    \item For every nonempty open sets $U,V$ there is $k$ such that for every $m$ there are $x\in U$ and $k_m$ with $T^{k_m+jk}(x)\in V$ for every $0\leq j\leq m$ and
    \item $T$ is hypercyclic and for every open set $U$ there is $k$ such that  $\cap_{j=1}^m T^{-jk}(U)\neq \emptyset $ for every $m$.
\end{enumerate}
\end{corollary}

 
If the operator is not weak$^*$-weak$^*$ continuous we still can characterize chaos in therms of the behavior of a single orbit.
\begin{proposition}[A characterization in terms of a single orbit]
Let $X$ be  separable  Fr\'echet space.
The following are equivalent:

\begin{enumerate}
    \item There exists a hypercyclic vector $x$ such that for every open set $U$ there is $(a_n)_n\sub N_T(x,U)\cap \kAP$,  such that $(T^{a_n}x)_n$ is is contained in a weakly compact set of $X$.
    \item $T$ is chaotic.
\end{enumerate}
\end{proposition}
\begin{proof}
We only prove (1) $\Longrightarrow $(2), being (2) $\Longrightarrow $(1) inmediate.

(1) $\Longrightarrow $(2). By Corollary \ref{dense small w-compact sets}, it suffices to show that $T$ has dense small weakly compact periodic sets. So let $U$ be an open set and consider  $V\sub \overline V\sub U$ a  convex open set. By assumption there is $k>0$ and a sequence $(k_n)_n$ such that for each $n,$  $T^{k_n+ik}x\in V$ for every $i\leq n$ and such that $K:=\{T^{k_n+ik}x: i\leq n\}$ is weakly precompact. Let $y$ be a weak accumulation point of $\{T^{k_n}x:  n\in\zN\}\subset K$. Then $y\in \overline K ^\omega$. Since $V$ is convex it follows that $y\in U$. 
Proceeding as in the proof of Lemma \ref{AP* implica dense small} (but using that $T$ is weak-weak continuous) we prove that $Y=\overline {Orb_{T^k}(y)}^\omega$ is a periodic set contained in $U$.  
Moreover, $Y$ is weakly compact because $T^{mk}(y)\in \overline K ^\omega $ for every $m$.
\end{proof}

\subsection{ Weighted shifts}\label{weighted shifts}
In this subsection we show that every hypercyclic weighted shift with dense small periodic sets is chaotic. We also show the existence of a weighted shift operator on $c_0$ that is $\kAP$-hypercyclic but does not have dense small periodic and hence it is not chaotic.


It is well known \cite[Theorem 8]{Gro00} that a backward shift defined on a Fr\'echet space with  unconditional basis $\{e_n\}$ is chaotic if and only if 
\begin{equation}\label{Chaos Weighted}
    \sum_{n=1}^\infty  e_n\in X.
\end{equation} 

\begin{theorem}\label{Backward c_0}
Let $\{e_n\}_n$ be an unconditional basis on a  Fr\'echet space $X$ and let $B:X\to X$ be the backward shift defined in $\{e_n\}$. Then $B$ is chaotic if and only if it has dense small periodic sets.
\end{theorem}
\begin{proof}
 By \eqref{Chaos Weighted} it is enough to show that $\sum_{n=1}^\infty e_{n}$ is convergent.

Let $\rho$ be a continuous seminorm such that for every $x,$ $|x_1|\leq \rho(x)$.
Since $B$ has dense small periodic sets there is $k\in \zN$ and $x$ such that  $B^{kn}(x)\in \f{1}{4}\{y: \rho (y)<1\}+e_1$ for every $n\geq 0$.  Thus $|x_1|\ge 1-\frac{1}{4}$ and  $\rho\l B^{nk}(x)-x\r< \f{1}{2}$ for every $n\in \zN$.  Then we have that  $|x_1- x_{1+nk}|=|e_1^*(x-B^{nk}(x))|< \f{1}{2}$ for every $n$.  Thus, $x_{1+nk}=  (x_1+\delta_n)$, where $\delta_n$ is a number of modulus less than $\f{1}{2}$. Note that, in particular we get that  $x_{1+nk}\neq 0$ for every $n$.
 
We consider now the series  $\sum_{n=1}^\infty e_{1+nk}=\sum_{n=1}^\infty \f{1}{x_1+\delta_n} x_{1+nk} e_{1+nk},$ which is (unconditionally) convergent by the unconditionallity of $\{e_n\}$. Finally we notice that 
$$
\sum_{n=1}^\infty  e_{n}=\sum_{j=0}^{k-1}\sum_{n\ge1}  e_{1+nk-j} =\sum_{j=0}^{k-1}B^j(\sum_{n\ge1} \f{1}{x_1+\delta_n} x_{1+nk} e_{1+nk}),$$
which is convergent. Therefore  $B$ is chaotic.
\end{proof}

\begin{corollary}
Let $\{e_n\}_n$ be an unconditional basis on a  Fr\'echet space $X$ and let $B_\omega:X\to X$ be a weighted backward shift defined in $\{e_n\}$. Then $B_\omega$ is chaotic if and only if it has dense small periodic sets.
\end{corollary}

 

On the other hand we show next that there are weighted backward shifts on $c_0$ that are $\kAP$-hypercylic but that are neither upper frequently hypercyclic nor chaotic. 
In \cite{Bes16} the authors exhibited an example of a reiterative hypercyclic weighted shift on $c_0$ that is not upper frequently hypercyclic. 
 A closer look to their proof shows in fact that their operator is $\kAP$-hypercyclic.
\begin{theorem}\label{Without periodic sets}
Let $S=\bigcup_{l,j} [l 10^j-j, l10^j+j]$ and $(w_n)$ the sequence of weights defined by $$w_n=\begin{cases}
2 & \text{ if } n\in S\\
\prod_{l=1}^{n-1} w_l^{-1} & \text{ if } n\in S+1\setminus S\\
1&\text {else}.
\end{cases}$$
Then $T:=B_\omega:c_0\to c_0$ is $\kAP$-hypercyclic and has not dense small periodic points. In particular it is not chaotic.
\end{theorem}

 The main argument used by the authors to prove that $T$ is reiterative hypercyclic is that $T$ satisfies the $\mathcal F$-hypercyclicity criterion applied to the family of sets with upper positive Banach density. Let us recall the criterion restricted to weighted shifts on $\ell_p$ or $c_0$.

\begin{theorem}[B\`es, Menet, Peris, Puig]
Let $\mathcal F$ be a Furstenberg family such that there exist  disjoint sets  $(A_k)_k\sub \mathcal F$  such that
\begin{itemize}
    \item [i)]  for any $j\in A_k$, any $j'\in A_{k'}$, $j\neq j'$ we have that $|j-j'|\geq \max\{k,k'\}$;
    \item [ii)]for any $k'\geq 0$ and any $k>k'$
    $$\sum_{n\in A_k+k'} \f{e_n}{\prod_{v=1}^n w_v}\in X \text{ and} \sum_{n\in A_k+k'} \f{e_n}{\prod_{v=1}^n w_v} \xrightarrow {k\to \infty}     0;$$
    \item [iii)] There are  $(C_{k,l})_{k,l}$ such that for every $k' \geq 0$, any $k>k'$ and any $l\geq1$,
    $$ \sup_{j\in A_l} \left\|\sum_{n\in A_k-j}\f{e_{n+k'}}{\prod_{v=1}^n w_{v+k'}}\right\| \leq C_{k,l}$$
    and such that $\sup_l C_{k,l}\to 0$ when $k\to \infty$ and such that for any $k$, $C_{k,l}\to 0$ when $l\to \infty$.
\end{itemize}
Then $B_w$ is $\mathcal F$-hypercyclic in $X=\ell_p$ or $c_0$.
\end{theorem}
\begin{proof}[Proof of Theorem \ref{Without periodic sets}] B\`es et. al. \cite{Bes16} proved that the operator is not upper frequently hypercyclic. Since a weighted backward shift on $c_0$ is chaotic if and only if it has dense small periodic sets and since chaotic weighted backward shifts are upper frequently hypercyclic, we conclude that the operator does not have dense small periodic sets.  

In \cite{Bes16}, sets $(A_k)_k$ satisfying $i-iii)$ of the above criterion and of positive Banach density were constructed. To prove that $T$ is $\kAP$-hypercyclic, it suffices to show that the sets $(A_k)_k$ chosen by the authors belong to $\kAP$.

Each $A_k$ is defined as $\cup_{j\in \phi^{-1}(k)} F_{j}$, where the  $\phi^{-1}(k)$ are disjoint infinite subsets of $\zN$ and the $F_{j}$ are defined as  $F_{j+1}:=\{10^{j_0}+10^{2k}l:0\leq l\leq l_0 \}$, where $l_0>j$ and $j_0$ is large enough (it is defined inductively). Thus, for each $j\in \phi^{-1}(k)$  each set $F_{j+1}$ is an arithmetic progression of length greater than $j$ with step $10^{2k}$. Since the set $\phi^{-1}(k)$ is infinite, we conclude that the sets $A_k$ have arbitrarily long arithmetic progressions with fixed step $10^{2k}$. 
\end{proof}

\subsection{The spectrum of an $\kAP$-hypercyclic operator}

In this subsection we study the spectrum of  $\kAP$-hypercyclic operators. Recall that  chaotic operators are easily seen to have perfect spectrum. In \cite{shk09}, Shkarin presented a very ingenious argument to prove that  frequently hypercyclic operators share the same property. We will see  that $\kAP$-hypercyclic operators  also have perfect spectrum.

Recall that an operator is said to be quasinilpotent provided that $\|T^n\|^\f{1}{n}\to 0.$
The proof of the next lemma is a modification of an analogous result for frequently hypercyclic operators (see \cite{shk09} or \cite[Lemma 9.38]{GroPer11}).
\begin{lemma}\label{lema qNIL}
	Let $S$ be an operator, $x^*\in X^*\setminus\{0\}$ and $U=\{y: Re(\left\langle y,x^*\right\rangle)>0,Re(\left\langle S(y),x^*\right\rangle)<0\}$. Suppose that for some $x\in U\setminus ker(x^*)$, \begin{equation}\label{eq lema espectro SALAP}
	\liminf_{k\to \infty}\frac{|N_S(x,U)\cap[0,k]|}{k+1}=\mu>0.    
	\end{equation}
	Then $S-I$ is not quasinilpotent.
	\end{lemma}
	
	\begin{proof}
Replacing $x^*$ by $\f{x^*}{Re(\langle x,x^*\rangle)}$ we can suppose that $Re(\langle x,x^*\rangle)=1$.

Suppose that $S-I$ is quasinilpotent. Then, given $\varepsilon>0$, there is some constant $M>0$ such that $\|(S-I)^k\|\le M\varepsilon^k,$ for every $k$. Thus we have for $z\in \mathbb C,$ and $|z|\le R$ that
\begin{align*}
\sum_{k=0}^\infty |Re(\langle (S-I)^kx,x^*\rangle )|\left|\frac{z(z-1)\dots(z-k+1)}{k!}\right| &\le M\|x\|\|x^*\| \sum_{k=0}^\infty \varepsilon^k\frac{R(R+1)\dots(R+k-1)}{k!} \\
&= \frac{M\|x\|\|x^*\|}{(1-\varepsilon)^R} 
\end{align*}
where we have used the generalized binomial theorem.

This implies that 
$$
f(z)=\sum_{k=0}^\infty Re(\langle (S-I)^kx,x^*\rangle)\frac{z(z-1)\dots(z-k+1)}{k!}
$$
defines an entire function of exponential type 0, such that $f(0)=\langle x,x^*\rangle=1$. Therefore, as a consequence of Jensen's formula, the number or zeros on the disk $\{|z|<R\}$, $n(R)$ is bounded above by
$$
\frac{\log(M\|x\|\|x^*\| (1-\varepsilon)^{-2R})}{\log 2}= c-2R\frac{\log(1-\varepsilon)}{\log 2}.
$$
Thus, we have
$$
\frac{n(k+1)}{k+1}\le \frac{c}{k+1}-\frac{2 (k+1)\log(1-\varepsilon)}{(k+1)\log 2} \to -\frac{2\log(1-\varepsilon)}{\log 2}.        
$$
This contradicts \eqref{eq lema espectro SALAP} because $\varepsilon$ can be chosen arbitrarily close to 0, and $|N_S(x,U)\cap[0,k]|\le n(k+1)$. Indeed, since 
\begin{align*}
f(n) &=\sum_{k=0}^n Re(\langle (S-I)^kx,x^*\rangle)\frac{n(n-1)\dots(n-k+1)}{k!}\\ 
&=  Re(\langle \sum_{k=0}^n \binom{n}{k}(S-I)^k I^{n-k}x,x^*\rangle) =Re( \langle S^nx,x^*\rangle),
\end{align*}
we have that $n\in N_S(x,U)$ if and only if $f(n)>0$ and $f(n+1)<0$. Finally, since $f|_{\mathbb R}$ is real valued, $f$ must have at least a zero in the open interval $(n,n+1)$.
\end{proof}

 We show now that $\kAP$-hypercyclic operators satisfy Ansari's property. Note that since $\kAP$ is neither Ramsey nor CuSP then this does not follow from \cite{shk09,bonilla2020frequently}.
 \begin{proposition}
 Let $T$ be an $\kAP$-hypercyclic operator. Then for every natural number $p$ we have that $T^p$ is $\kAP$-hypercyclic. Moreover they share the $\kAP$-hypercyclic vectors.
 \end{proposition}
 \begin{proof}
	Since $T$ is hypercyclic, it follows by Ansari's Theorem  that $T^p$ is hypercyclic. Since, by Proposition \ref{equivalencias salap} for each open set $U$ there is $k$ such that for every $m$ , $\bigcap_{j=1}^m T^{-jk}(U)\neq \emptyset,$ we have that  for every $m$ , $\bigcap_{j=1}^m T^{-jkp}(U)\neq \emptyset.$ Applying again Proposition \ref{equivalencias salap}, we are done.
\end{proof}
Note that by Proposition \ref{equivalencias salap} every hypercyclic vector of an $\kAP$-hypercyclic operator is an $\kAP$-hypercyclic vector and by Ansari's Theorem the hypercyclic vectors of $T$ and $T^p$ coincide. We conclude that every hypercyclic vector of $T$ must be an $\kAP$-hypercyclic vector of $T^p$. 
On the other hand, the rotations of  $\kAP$-hypercyclic operators need not to be $\kAP$-hypercyclic.  
\begin{remark}
There are an $\kAP$-hypercyclic operator and $\lambda\in\zT$ such that $\lambda T$ is not $\kAP$-hypercyclic.
\end{remark}
\begin{proof}
 It is known that there are a chaotic operator $T$ and $\lambda\in \zT$ in a Hilbert space  such that $\lambda T$ is not chaotic, see \cite{BayBer09}. Hence, $T$ is $\kAP$-hypercyclic and by Theorem \ref{Caoticos}, $\lambda T$ is not $\kAP$-hypercyclic.
\end{proof}

 \begin{theorem}\label{teo qNIL}
Let $T$ be a $\kAP$-hypercyclic operator on a $\mathbb K$-Banach space. Then $T-\lambda Id$ is not quasinilpotent for any $|\lambda|=1$.
\end{theorem}


\begin{proof}
	Let $\lambda=e^{2\pi i\theta}$. Suppose first that  $\theta=\frac{p}{q}$ a rational angle.
	
	Note that if $T-\lambda I$ quasinilpotent, $q\in \mathbb N$ then $(T^{**})^q-\lambda^qI=(T^q)^{**}-I$ is a quasinilpotent operator on $X^{**}$.
	
	We will apply the above lemma for $S=(T^q)^{**}$. Let $x^*\in X^*\setminus\{0\}$ and $U=\{y\in X^{**}: Re(\left\langle y,x^*\right\rangle)>0,Re(\left\langle S(y),x^*\right\rangle)<0\}$.
	Note that since $T^q$ is hypercyclic, $U\neq \emptyset$ and, moreover, it contains a nonempty open ball $V$ of $X$ such that $\interior{\overline V}^{\omega^*}\subset U$.
	Then, since $T^q$ is $\kAP$-hypercyclic, Proposition \ref{prop AP* implica ptos periodicos en el bidual} implies that there are $x\in \interior{\overline V}^{\omega^*}$ and $m\in \zN$ for which $S^{jm}(x)\in \interior{\overline V}^{\omega^*}$ for every $j\in \zN$. In particular $\liminf_{k\to \infty}\frac{|N_S(x,U)\cap[0,k]|}{k+1}\ge \frac1{m}>0$.

	Therefore we have that $S-I$ and hence $T-\lambda I$ is not quasinilpotent for $|\lambda|=1$ with rational angle.

Suppose suppose now that $\theta$ is an irrational angle. Note that it suffices to prove that $S-I$ is not quasinilpotent, where $S=e^{-2\pi i\theta}T^{**}$.


Let $U$ be the open set of $X^{**}$ defined in Lemma \ref{lema qNIL} for $S$. Since $e^{-2\pi i\theta}T$ is hypercyclic, 
$U$ is non-empty and contains a nonempty open ball $V$ of $X$ such that $\interior{\overline V}^{\omega^*}\subset U$. 

For small $\delta>0$ let $V_\delta:=\{x\in V:\, d(x,V^c)>\delta\textrm{ and }\|x\|<1/\delta\}$. Since $T$ is $\kAP$-hypercyclic, by Proposition \ref{prop AP* implica ptos periodicos en el bidual}, there is some $x\in \interior{\overline{ V_\delta}}^{\omega^*}$ and $m$ such that $T^{jm}x\in \interior{\overline{ V_\delta}}^{\omega^*}\subset U$ for every $j\in\mathbb N.$

We claim that if $\varepsilon<\delta^2/4\pi$, $y\in \interior{\overline{ V_\delta}}^{\omega^*}$ and $\varphi\in p+(-\varepsilon,\varepsilon)$ for some $p\in\mathbb Z$, then $e^{2\pi i \varphi}y\in \interior{\overline{ V}}^{\omega^*}\subset U$. Indeed, if $z\notin \interior{\overline{ V}}^{\omega^*}$,
$$
\|e^{2\pi i \varphi}y-z\|\ge\|y-z\|-\|y(1-e^{2\pi i \varphi})\|\ge \delta-\frac1{\delta}\varepsilon 2\pi>\delta/2.
$$
	
Define now
$$
A:=\{j:-jm\theta\in p+(-\varepsilon,\varepsilon)\textrm{ for some }p\in\mathbb Z\}.
$$ 
Since $m\theta$ is irrational, $\underline {dens}(A)>0$, and by the claim, 
$$
A\subset \{j:S^{mj}x\in U \}.
$$	
Thus
$$
0<\underline{dens}(A)\le m\cdot \underline{dens}(N_S(x,U)).
$$	
Therefore by Lemma \ref{lema qNIL}, $S-I$ is not quasinilpotent. 
\end{proof}

\begin{corollary}\label{espectro}
The spectrum of an $\kAP$-hypercyclic operator cannot have isolated points.
\end{corollary}
\begin{proof}
If $\lambda$ is an isolated point of the spectrum of a hypercyclic operator then by the Riesz decomposition Theorem and the fact that the property of having dense small periodic sets is preserved under quasiconjugacies,  we may construct an operator $S$ having dense small periodic sets and such that $\sigma(T)=\lambda$. Since $T$ is hypercyclic, it would be of the form $T=S+\lambda I$ for some $|\lambda|=1$. By the spectral radius formula, $S$ would be quasinilpotent, contradicting Theorem \ref{teo qNIL}.
\end{proof}

\begin{corollary}
There are not  $\kAP$-hypercyclic operators on hereditarily indecomposable Banach spaces.
\end{corollary}

\section{Final comments and questions}
We would like to end this note with some questions related with the results discussed in the preceding paragraphs.  

 The proof of Theorem \ref{Caoticos} relies on the normabilty of the space. 
\begin{question}
Does Theorem \ref{Caoticos} hold on arbitrary Fr\'echet spaces?
\end{question}
In Theorem \ref{Without periodic sets} we showed the existence of an $\kAP$-hypercyclic operator that is not chaotic. By  Theorem \ref{Backward c_0} the operator does not have dense small periodic sets. In fact, we did not come to an operator that has dense small periodic points and does not have dense periodic points.
\begin{question}\label{pregunta small periodic sets implica caos}
Is any hypercyclic operator with dense small periodic sets necessarily chaotic. Or, more generally, does any operator with dense small periodic sets have dense periodic points?
\end{question}
We answered Question \ref{Pregunta1} for a wide class of operators and spaces. However the general question whether there exists a Furstenberg family $\A$ for which $\A$-hypercyclicity is equivalent to chaos remains open.


The following diagram shows the known implications between the concepts appearing in this article. A solid arrow means that the implication holds.
A dashed arrow means that the implication holds with some extra hypothesis (here in both cases weak$^*$-weak$^*$-continuity of the operator suffices). For the dotted line we don't know if the implication holds in general (Questions \ref{pregunta small periodic sets implica caos}) and all other implications are known to fail. 

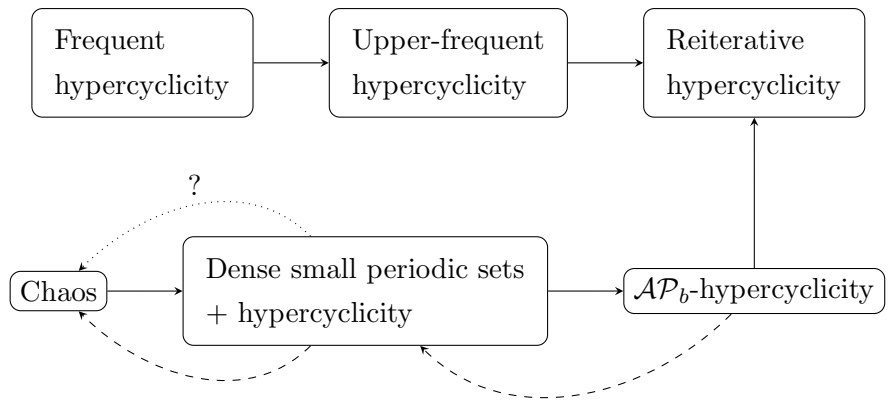
\begin{center}
\begin{tikzpicture}[>=stealth,every node/.style={shape=rectangle,draw,rounded corners},scale=0.6]
    \node (F) {\begin{tabular}{l} Frequent \\ hypercyclicity \end{tabular}};
    \node (UF) [right=of F]{\begin{tabular}{l} Upper-frequent \\ hypercyclicity \end{tabular}};
    \node (R) [right =of UF]{\begin{tabular}{l} Reiterative \\ hypercyclicity \end{tabular}};
    
    \node (kAP) [below=2cm of R]{$\kAP$-hypercyclicity};
    \node(DS)[left =of kAP]{\begin{tabular}{l} Dense small periodic sets\\+ hypercyclicity \end{tabular}};
    \node (CHA) [left =of DS]{Chaos};

    \draw[->] (F) to[out=0,in=180] (UF);
     \draw[->] (UF) -- (R);
     \draw[->] (kAP) -- (R);
     \draw[->] (CHA) -- (DS);
    \draw[->] (DS) -- (kAP);

    \draw[->,dashed]  (kAP) to [out=225,in=315]  (DS);
     \draw[->,dashed]  (DS) to [out=225,in=315]   (CHA);
 \draw[->,dotted]  (DS) to [out=135,in=45] node[above,draw=none]  {?}  (CHA);



  
   
\end{tikzpicture}
\end{center}
\bibliography{biblio}

\begin{thebibliography}{10}

\bibitem{Aki97}
E.~Akin.
\newblock {\em Recurrence in topological dynamics}.
\newblock The University Series in Mathematics. Plenum Press, New York, 1997.
\newblock Furstenberg families and Ellis actions.

\bibitem{ansari1997existence}
S.~I. Ansari.
\newblock Existence of hypercyclic operators on topological vector spaces.
\newblock {\em {J}ournal of {F}unctional {A}nalysis}, 148(2):384--390, 1997.

\bibitem{BayBer09}
F.~Bayart and T.~Berm{\'u}dez.
\newblock Semigroups of chaotic operators.
\newblock {\em Bulletin of the London Mathematical Society}, 41(5):823--830,
  2009.

\bibitem{BayGri06}
F.~Bayart and S.~Grivaux.
\newblock Frequently hypercyclic operators.
\newblock {\em Trans. Amer. Math. Soc.}, 358(11):5083--5117 (electronic), 2006.

\bibitem{BayGri07}
F.~Bayart and S.~Grivaux.
\newblock Invariant {G}aussian measures for operators on {B}anach spaces and
  linear dynamics.
\newblock {\em Proc. Lond. Math. Soc. (3)}, 94(1):181--210, 2007.

\bibitem{BayMat09}
F.~Bayart and E.~Matheron.
\newblock {\em Dynamics of linear operators}, volume 179 of {\em Cambridge
  Tracts in Mathematics}.
\newblock Cambridge University Press, Cambridge, 2009.

\bibitem{bernal1999hypercyclic}
L.~Bernal-Gonz{\'a}lez.
\newblock On hypercyclic operators on {B}anach spaces.
\newblock {\em Proceedings of the American Mathematical Society},
  127(4):1003--1010, 1999.

\bibitem{Bes16}
J.~B{\`e}s, Q.~Menet, A.~Peris, and Y.~Puig.
\newblock Recurrence properties of hypercyclic operators.
\newblock {\em Mathematische Annalen}, 366(1-2):545--572, 2016.

\bibitem{BesMen19}
J.~B{\`e}s, Q.~Menet, A.~Peris, and Y.~Puig.
\newblock Strong transitivity properties for operators.
\newblock {\em Journal of Differential Equations}, 266(2-3):1313--1337, 2019.

\bibitem{BonMarPer01}
J.~Bonet, F.~Mart{\'\i}nez-Gim{\'e}nez, and A.~Peris.
\newblock A banach space which admits no chaotic operator.
\newblock {\em Bulletin of the London Mathematical Society}, 33(2):196--198,
  2001.

\bibitem{BonGro18}
A.~Bonilla and K.-G. Grosse-Erdmann.
\newblock Upper frequent hypercyclicity and related notions.
\newblock {\em Revista Matem{\'a}tica Complutense}, 31(3):673--711, 2018.

\bibitem{bonilla2020frequently}
A.~Bonilla, K.-G. Grosse-Erdmann, A.~L{\'o}pez-Mart{\'\i}nez, and A.~Peris.
\newblock Frequently recurrent operators.
\newblock {\em arXiv preprint arXiv:2006.11428}, 2020.

\bibitem{dlRRea09}
M.~de~la Rosa and C.~Read.
\newblock A hypercyclic operator whose direct sum {$T\oplus T$} is not
  hypercyclic.
\newblock {\em J. Operator Theory}, 61(2):369--380, 2009.

\bibitem{Dev89}
R.~L. Devaney.
\newblock {\em An introduction to chaotic dynamical systems}.
\newblock Studies in Nonlinearity. Westview Press, Boulder, CO, 2003.
\newblock Reprint of the second (1989) edition.

\bibitem{GodSha91}
G.~Godefroy and J.~H. Shapiro.
\newblock Operators with dense, invariant, cyclic vector manifolds.
\newblock {\em J. Funct. Anal.}, 98(2):229--269, 1991.

\bibitem{Gro00}
K.~Grosse-Erdmann.
\newblock Hypercyclic and chaotic weighted shifts.
\newblock {\em Studia Math}, 139(1):47--68, 2000.

\bibitem{GroPer11}
K.-G. Grosse-Erdmann and A.~Peris~Manguillot.
\newblock {\em {Linear chaos.}}
\newblock {Universitext. Springer, London}, 2011.

\bibitem{HuaYe05}
W.~Huang and X.~Ye.
\newblock Dynamical systems disjoint from any minimal system.
\newblock {\em Trans. Amer. Math. Soc.}, 357(2):669--694, 2005.

\bibitem{KreiSmul40}
M.~Krein and V.~{\v S}mulian.
\newblock On regularly convex sets in the space conjugate to a banach space.
\newblock {\em Annals of Mathematics}, pages 556--583, 1940.

\bibitem{Li11}
J.~Li.
\newblock Transitive points via {F}urstenberg family.
\newblock {\em Topology Appl.}, 158(16):2221--2231, 2011.

\bibitem{Men17}
Q.~Menet.
\newblock Linear chaos and frequent hypercyclicity.
\newblock {\em Transactions of the American Mathematical Society},
  369(7):4977--4994, 2017.

\bibitem{Pui17}
Y.~Puig.
\newblock Frequent hypercyclicity and piecewise syndetic recurrence sets.
\newblock {\em arXiv preprint arXiv:1703.09172}, 2017.

\bibitem{SchaeferTVS}
H.~H. {Schaefer} and M.~P. {Wolff}.
\newblock {\em {Topological vector spaces. 2nd ed.}}, volume~3.
\newblock New York, NY: Springer, 2nd ed. edition, 1999.

\bibitem{shk09}
S.~Shkarin.
\newblock On the spectrum of frequently hypercyclic operators.
\newblock {\em Proc. Amer. Math. Soc.}, 137(1):123--134, 2009.

\bibitem{Tyc35}
A.~Tychonoff.
\newblock Ein fixpunktsatz.
\newblock {\em Mathematische Annalen}, 111(1):767--776, 1935.

\end{thebibliography}

\bibliographystyle{abbrv}

\end{document}